\documentclass[psamsfonts]{amsart}

\usepackage{amssymb,amsfonts}
\usepackage[all,arc]{xy}
\usepackage{enumerate}
\usepackage{mathrsfs}
\usepackage{mathtools}
\DeclarePairedDelimiterX\set[1]\lbrace\rbrace{#1}
\newtheorem{theorem}{Theorem}[section]
\newtheorem{corollary}[theorem]{Corollary}
\newtheorem{proposition}[theorem]{Proposition}

\theoremstyle{definition}
\newtheorem{definition}[theorem]{Definition}

\newtheorem{example}[theorem]{Example}

\theoremstyle{remark}

\let\phi=\varphi

\let\oldbigwedge\bigwedge
\def\BIGwedge{{\textstyle\oldbigwedge}}
\def\medwedge{{\scriptstyle\oldbigwedge}}
\def\bigwedge{\mathchoice{\BIGwedge}{\BIGwedge}{\medwedge}{}}

\makeatletter

\let\epsilon=\varepsilon

\let\c@equation\c@theorem
\makeatother
\numberwithin{equation}{section}

\begin{document}
\title{Two-layered numbers}
    
\author{\bfseries H. Behzadipour}
    
\address{School of Electrical and Computer Engineering\\ University College of Engineering\\ University of Tehran\\ Tehran\\ Iran}
    
\email{hussein.behzadipour@gmail.com, h.behzadi@ut.ac.ir}

    

\begin{abstract}
In this paper, first, I introduce two-layered numbers. Two-layered numbers are positive integers that their positive divisors except $1$ can be partitioned into two disjoint subsets. Similarly, I defined a half-layered number as a positive integer $n$ that its proper positive divisors excluding $1$ can be partitioned into two disjoint subsets. I also investigate the properties of two-layered and half-layered numbers and their relation with practical numbers and Zumkeller numbers.

\end{abstract}
    
\maketitle
    
\section{Introduction}

A perfect number is a positive integer $n$ that equals the sum of its proper positive divisors. Generalizing the concept of perfect numbers, Zumkeller in \cite{zumkeller} published a sequence of integers that their divisors can be partitioned into two disjoint subsets with equal sum. Clark et al. in \cite{clark} called such integers Zumkeller numbers and investigated some of their properties, and also suggested some conjectures about them. Peng and Bhaskara Rao in \cite{rao} introduced half-Zumkeller numbers and provided interesting results about Zumkeller numbers.
  
  In the present paper, I define two-layered numbers based on the concept of perfect numbers and Zumkeller numbers.  A two-layered number is a positive integer $n$ that its positive divisors excluding $1$ can be partitioned into two disjoint subsets of an equal sum. A partition $\set {A, B}$ of the set of positive divisors of $n$ except $1$ is a two-layered partition if each of $A$ and $B$ has the same sum.
  
  In the first section, I investigate the properties of two-layered numbers. For a two-layered number $n$, that sum of its divisors is $\sigma(n)$, the following statements hold (See Proposition \ref{sigmaodd}):
  
  Let $\sigma(n)$ be the sum of all positive divisors of $n$. If $n$ is a two-layered number, then
      \begin{enumerate}
      \item $\sigma(n)$ is odd.
      \item Powers of all odd prime factors of $n$ should be even.
      \item $\sigma(n) \geq 2n+1$, so $n$ is abundant.
      \end{enumerate}
      
      After that, In theorem \ref{twolayprop}, I prove that The integer $n$ is a two-layered number if and only if $\frac{\sigma(n)-1}{2}-n$ is a sum of distinct proper positive divisors of n excluding 1. I also introduce two methods of generating new two-layered numbers from known two-layered numbers. Suppose that $n$ is a two-layered number and $p$ is a prime number with $(n,p)=1$, then $np^\alpha$ is a two-layered number for any even positive integer $\alpha$ (See Theorem \ref{npalpha}). We can also generate two-layered numbers in another way. Let $n$ be a two-layered number and $p_1^{k_1}p_2^{k_2} \dots p_m^{k_m}$ be the prime factorization of $n$. Then for any nonnegative integers $\alpha_1, \dots \alpha_m$, the integer $$p_1^{k_1+\alpha_1(k_1+1)}p_2^{k_2+\alpha_2(k_2+1)} \dots p_m^{k_m+\alpha_m(k_m+1)}$$ is a two-layered number (See Theorem \ref{secondwaygenerate}).
      
      In the second section of the present paper, I generalize the concept of practical numbers and define semi-practical numbers. A practical number is a positive integer $n$ that every positive integer less than $n$ can be represented as a sum of distinct positive divisors of $n$ \cite{practical}. A positive integer $n$ is a semi-practical number if every positive integer $x$ where $1<x<n$ can be represented as a sum of distinct positive divisors of $n$ excluding $1$ (See Definition \ref{defsemipractical}). 
      
      I investigate some properties of semi-practical numbers and their relations with two-layered numbers. For example, every semi-practical number is divisible by $12$ (See Proposition \ref{divis12}). I also proved that a positive integer $n$ is is a semi-practical number if and only if every positive integer $x$ where $1<x<\sigma(n)$, is a sum of distinct positive divisors of $n$ excluding $1$ (See Theorem \ref{spracsigma}). The most important relation between semi-practical numbers and two-layered numbers is that a semi-practical number $n$ is two-layered if and only if $\sigma(n)$ is odd (See Proposition \ref{keyprop}).
      
      In section 3, I define a half-layered number.     A positive integer $n$ is said to be a half-layered number if the proper positive divisors of $n$ excluding $1$ can be partitioned into two disjoint non-empty subsets of an equal sum (See Definition \ref{halflay}). A half-layered partition
      for a half-layered number $n$ is a partition $\set {A, B}$ of the set of proper positive divisors of $n$ excluding $1$ so that
      each of $A$ and $B$ sums to the same value (See Definition \ref{defhalflayeredpartition}).
  
      After these definitions, I investigate the properties of half-layered numbers. For example, A positive integer $n$ is half-layered if and only if $\frac{\sigma(n)-n-1}{2}$ is the sum of some distinct positive proper positive divisors of $n$ (See Proposition \ref{sigma2n}). A positive even integer $n$ is half-layered if and only if $\frac{\sigma(n)-2n-1}{2}$ is the sum (possibly empty
      sum) of some distinct positive divisors of $n$ excluding $n$, $\frac{n}{2}$, and $1$ (See Theorem \ref{halflay}). If $n$ is an odd half-layered number, then at least one of the powers of prime factors of $n$ should be even (See Proposition \ref{oddhalflayered}).
      
      Using the definition of half-Zumkeller numbers, we can derive some of the interesting properties of half-layered numbers. A positive integer $n$ is said to be a half-Zumkeller number if the proper positive divisors of $n$ can be partitioned into two disjoint non-empty subsets of an equal sum. A half-Zumkeller partition for a half-Zumkeller number n is a partition $\set{A, B}$ of the set of proper positive divisors of $n$ so that each of $A$ and $B$ sums to the same value (Definition 3 in \cite{rao}). Based on these definition, I prove that if $m$ and $n$ are half-layered numbers with $(m, n)=1$, then $mn$ is half-layered (See Proposition \ref{mnhalflayered}).
      
      After that, I investigate some relations between half-layered and two-layered numbers. For example, let $n$ be even. Then $n$ is half-layered if and only if $n$ admits a two-layered partition
      such that $n$ and $\frac{n}{2}$ are in distinct subsets. Therefore, if $n$ is an even half-layered number then $n$ is two-layered (See Proposition \ref{halfzumnn2}). It is also proved that if $n$ is an even two-layered number and If $\sigma(n) < 3n$, then $n$ is half-layered (See Theorem \ref{itisalsoproved}).     Let $n$ be even. Then, $n$ is two-layered if and only if either $n$ is half-layered or $\frac{\sigma(n)-3n-1}{2}$ is a sum
      (possibly an empty sum) of some positive divisors of $n$ excluding $n$, $\frac{n}{2}$, and $1$ (See Proposition \ref{sigma3n}).
      
      If $6$ divides $n$, $n$ is two-layered, and $\sigma(n) < \frac{10n}{3}$ , then $n$ is half-layered (See Proposition \ref{6dividesn}). If $n$ is two-layered, then $2n$ is half-layered (See Proposition \ref{n2ntwolayered}). Let $n$ be an even half-layered number and $p$ be a prime with (n, p) = 1. Then $np^{\ell}$ is half-
      layered for any positive integer $\ell$ (See Proposition \ref{andprimefactorization}). Let $n$ be an even half-layered number and the prime factorization of $n$ be $ p_1^{k_1} p_2^{k_2} /dots p_m^{k_m} $ Then for nonnegative integers $\ell_1, \dots , \ell_m$, the integer
                $$ p_1^{k_1+\ell_1(k_1+1)} p_2^{k_2+\ell_2(k_2+1)} \dots p_m^{k_m+\ell_m(k_m+1)}    $$
                is half-layered (See Theorem \ref{numberandtheprime}).    

\section{two-layered numbers}
\label{section1}
	
    \begin{definition}
    A positive integer $n$ is a two-layered number if the positive divisors of $n$ excluding $1$ can be partitioned into two disjoint subsets of an equal sum.
    \end{definition}
    
    \begin{definition}
    A two-layered partition for a two-layered number $n$ is a partition $\set {A,B}$ of the set of positive divisors of $n$ excluding $1$ so that each of $A$ and $B$ sums to the same value.
    \end{definition}
    
    \begin{example}
    \label{example1}
    The number 36 is a two-layered number and its two-layered partition is $\set{A,B}$, where $A = \set{2, 3, 4, 36}$ and $B=\set {6, 9, 12, 18}$. You can check that each of $A$ and $B$ has the sum of $45$. The numbers $72, 144,$ and $200$ are also two-layered. You can find the sequence of two-layered numbers in \cite{twosequenc}.
    \end{example} 
    
    \begin{proposition}
    \label{sigmaodd}
    Let $\sigma(n)$ be the sum of all positive divisors of $n$. If $n$ is a two-layered number, then
    \begin{enumerate}
    \item $\sigma(n)$ is odd.
    \item Powers of all odd prime factors of $n$ should be even.
    \item $\sigma(n) \geq 2n+1$, so $n$ is abundant.
    \end{enumerate}
    \end{proposition}

    \begin{proof}
    $(1):$ If $\sigma(n)$ is even, then $\sigma(n)-1$ is odd, so it is impossible to partition the positive divisors of $n$ into two subset of equal sum.
    
    $(2):$ using $(1)$, the number of odd positive divisors of $n$ is odd. Suppose that the prime factorization of $n$ is $2^{k_0}p_1^{k_1}p_2^{k_2}\dots p_m^{k_m}$. The number of odd positive divisors of $n$ is $(k_1+1)(k_2+1)\dots (k_m+1)$. All of $k_i$ should be even in order to make the product $(k_1+1)(k_2+1)\dots (k_m+1)$ odd. 
    
    $(3):$ Let $n$ be a two-layered number with two-layered partition $\set {A,B}$. Without loss of generality we may assume that $n \in A$, so the sum in $A$ is at least $n$ and we can conclude $\sigma(n)-1 \geq 2n$.
    \end{proof}
    
    \begin{theorem}
    \label{twolayprop}
    The integer $n$ is a two-layered number if and only if $\frac{\sigma(n)-1}{2}-n$ is a sum of distinct proper positive divisors of n excluding 1.
    \end{theorem} 
    
    \begin{proof}
    Let $n$ be a two-layered number and its two-layered partition is $\set{A,B}$. Without loss of generality we assume that $n \in A$, so the sum of the remaining elements of $A$ is $\frac{\sigma(n)-1}{2}-n$.
    
    Conversely, if we have a set of proper divisors of $n$ excluding $1$ that its sum is $\frac{\sigma(n)-1}{2}-n$, we can augment this set with $n$ to construct a set of positive divisors of $n$ summing to $\frac{\sigma(n)-1}{2} $. The complementary set of positive divisors of $n$ sums to the same value, and so these two sets form a two-layered partition for $n$.
    \end{proof}
    
    With the help of the next two theorems, we can generate some new two-layered numbers by knowing a two-layered number.
    
    \begin{definition} [Definition 1 in \cite{rao}]
    A positive integer $n$ is said to be a Zumkeller number if the positive divisors of $n$ can be
    partitioned into two disjoint subsets of equal sum. A Zumkeller partition for a Zumkeller number $n$ is
    a partition $\set{A, B}$ of the set of positive divisors of $n$ so that each of $A$ and $B$ sums to the same value.
    \end{definition}
    
    \begin{theorem}
    \label{npalpha}
    Let $n$ be a two-layered number and $p$ be a prime number with $(n,p)=1$, then $np^\alpha$ is a two-layered number for any even positive integer $\alpha$.
    \end{theorem}
    
    \begin{proof}
    Suppose that $\set {A,B}$ is a Zumkeller partition of $n$. Then $\set {(A \setminus \set{1}) \cup (pA)\cup (p^2A)\cup \dots \cup (p^\alpha A), (B \setminus \set{1}) \cup (pB) \cup(p^2B) \cup \dots \cup (p^\alpha B)}$ is a two-layered partition of $np^\alpha$.
    \end{proof}
    
    \begin{theorem}
    \label{secondwaygenerate}
    Suppose that $n$ is a two-layered number and $p_1^{k_1}p_2^{k_2} \dots p_m^{k_m}$ is the prime factorization of $n$. Then for any nonnegative even integers $\alpha_1, \dots \alpha_m$, the integer $$p_1^{k_1+\alpha_1(k_1+1)}p_2^{k_2+\alpha_2(k_2+1)} \dots p_m^{k_m+\alpha_m(k_m+1)}$$ is a two-layered number. 
    \end{theorem}
    
    \begin{proof}
    If we show that $p_1^{k_1+\alpha_1(k-1+1)}p_2^{k_2} \dots p_m^{k_m}$ the proof will be completed. Suppose that $\set {A,B}$ is a Zumkeller partition of $n$. If $D$ is the set of positive divisors of $n$, then $(D \setminus \set {1}) \cup (p_1^{k_1+1} D) \cup (p_1^{2(k_1+1)}D) \cup \dots \cup (p_1^{\alpha_1(k_1+1)}D))$ is the set of positive divisors of  $p_1^{k_1+\alpha_1(k-1+1)}p_2^{k_2} \dots p_m^{k_m}$ excluding 1. Therefore a two-layered partition for  $p_1^{k_1+\alpha_1(k-1+1)}p_2^{k_2} \dots p_m^{k_m}$ is $\set {A \setminus \set {1} \cup (p_1^{k_1+1}A) \cup (p_1^{2(k_1+1)}A) \cup \dots \cup (p_1^{\alpha_1(k_1+1)}A), B \setminus \set {1} \cup (p_1^{k_1+1}B) \cup (p_1^{2(k_1+1)}B) \cup \dots \cup (p_1^{\alpha_1(k_1+1)}B)  }$ and the proof is complete.
    \end{proof}
    
\section{semi-practical numbers and two-layered numbers}
 Practical numbers have been introduced by Srinivasan in 1948 as what follows:  
    \begin{definition}
A positive integer $n$ is a practical number if every positive integer less than $n$ can be represented as a sum of distinct positive divisors of $n$.\cite{practical}
    \end{definition}
    
Because of the structure of two-layered number, if we change the definition of practical numbers and call them semi-practical numbers, we can drive some useful relation between them and two-layered numbers, so I define semi-practical numbers as what follows:

\begin{definition}
\label{defsemipractical}
A positive integer $n$ is practical if every positive integer $x$ where $1<x<n$ can be represented as a sum of distinct positive divisors of $n$ excluding $1$. 
\end{definition}    
    
    \begin{proposition}
    \label{divis12}
    Every semi-practical number is divisible by $12$.
    \end{proposition}
    \begin{proof}
    Since we can not write $2,3,$ and $4$ as sums of more than one positive integer greater than $1$, they should be divisors of our semi-practical number.
    \end{proof}
    
    \begin{theorem}
    \label{spracsigma}
    A positive integer $n$ is is a semi-practical number if and only if every positive integer $x$ where $1<x<\sigma(n)$, is a sum of distinct positive divisors of $n$ excluding $1$.
    \end{theorem}

    \begin{proof}
    Suppose that $n$ is a semi-practical number. I introduce an algorithm for writing all positive integer $x$ between $n$ and $\sigma(n)$ as sum of distinct positive divisors of $n$ excluding $1$.
    
    First, let $x$ be $n+1$. Since $n$ is semi-practical, by Propositin \ref{divis12}, it is divisible by $n/2$ and $n/3$. Hence, $n+1 = n/2 + n/3 + r$, where $r$ is a positive integer. By Proposition \ref{divis12}, $n > 6$, so $n+1-n/2-n/3 < n/3$. On the other hand, since $n$ is a semi-practical number and $r<n/3<n$, $r$ is equal to some of distinct divisors of $n$ which are less than $n/3$ and greater than $1$.
    
    For $n+1<x<\sigma(n)$, let the positive divisors of $n$ which are greater than $1$ be written in increasing order as $m_1<m_2 < \dots < m_k$. Now we can write $x= \sum_{i=\ell}^k m_i+r$ where $1\leq \ell \leq k$ and $0 \leq r < m_{\ell -1}$. If $r=0$ then $x$ is a sum of distinct divisors of $n$. If $1<r< m_{\ell -1}$, since $n$ is semi-practical and $r<n$, then we can write $r$ as a sum of distinct divisors of $n$ which are less than $m_{\ell-1}$, so $x$ is a sum of distinct divisors of $n$. If $r=1$, then we can write $x=\sum_{i=\ell+1}^k+r_1$ where $1<r_1<m_{\ell}$. since $n$ is semi-practical and $r<n$, then $r_1$ is sum of distinct divisors of $n$ which are less than $m_{\ell}$, so $x$ is a sum of distinct divisors of $n$.
    
    Conversely, if every positive integer less than $\sigma(n)$ excluding $1$, is a some of distinct positive divisors of $n$ excluding $1$, it is clear that $n$ is semi-practical.
    
    \end{proof}

    \begin{proposition}
    \label{keyprop}
    A semi-practical number $n$ is two-layered if and only if $\sigma(n)$ is odd.
    \end{proposition}
    
    \begin{proof}
    If $n$ is two-layered number, then $\sigma(n)$ is odd by Proposition \ref{sigmaodd}. Conversely, if $\sigma(n)$ is odd, then $\frac{\sigma(n)-1}{2}$ is a positive integer smaller than $\sigma(n)$. Since $n$ is a semi-practical number, using Proposition \ref{spracsigma}. 
    \end{proof}

    \begin{theorem}
    \label{somesum}
    Let $n$ be a positive integer and $p$ be a prime with $(n, p) = 1$. Let $D$ be the set of all positive
    divisors of $n$ including $1$. The following conditions are equivalent:
    \begin{enumerate}
    \item $np$ is two-layered.
    \item There exist two partitions $\set {D_1,D_2}$ and $\set{D_3, D_4}$ of $D \setminus \set{1}$ such that $$p(\sum_{d \in D_1}d-\sum_{d \in D_2}d)=(\sum_{d \in D_3}d-\sum_{d \in D_4}d).$$
    \item There exists a partition $\set {D_1,D_2}$ of $D\setminus \set{1}$ and subsets $A_1 \subseteq D_1$ and $A_2 \subseteq D_2$ such that $$\frac{p+1}{2}(\sum_{d \in D_1}d-\sum_{d \in D_2}d)=(\sum_{d \in A_1}d-\sum_{d \in A_2}d).$$
    \end{enumerate}
    \end{theorem}
    
    \begin{proof}
    It is clear that $(pD) \cup (D\setminus\set{1})$ is the set of all positive divisors of $np$ excluding $1$.
    
    $(1) \Rightarrow$ (2). Suppose that $np$ is two-layered. Hence, there is a two-layered partition $\set {A,B}$ of $(pD) \cup (D\setminus\set{1})$. Let $D_1=\frac{1}{p}(A \cap (pD))$, $D_2=\frac{1}{p}(B \cap (pD))$, $D_3=B \cap (D\setminus\set{1})$, $A \cap (D\setminus \set{1})$, then $$p \sum_{d \in D_1}d + \sum_{d \in D_4}d = p \sum_{d \in D_2}d + \sum_{d \in D_3}d.$$ and the proof is complete.
    
    $(2) \Rightarrow (3)$. Let $A_1 = D_1 \cap D_3$ and $A_2 = D_2\cap D_4$. We have

    \begin{align}
    \frac{p+1}{2}(\sum_{d \in D_1}d - \sum_{d \in D_2}d) & =  \frac{1}{2}[ p (\sum_{d \in D_1}d - \sum_{d \in D_2}d)+(\sum_{d \in D_1}d- \sum_{d \in D_2}d)] \nonumber \\
    &  = \frac{1}{2} [\sum_{d \in D_3}d - \sum_{d \in D_4}d + \sum_{d \in D_1}d - \sum_{d \in D_2}d] \nonumber\\
    &  = \frac{1}{2}[2 (\sum_{d \in D_1 \cap D_3}d) - 2 (\sum_{d \in D_2 \cap D_4}d)] \nonumber \\
    &  = \sum_{d \in A_1}d - \sum_{d \in A_2}d. \nonumber
    \end{align}
    
    $(3) \Rightarrow (1) $. We can rewrite the equation in $(3)$ as follows:
    $$\frac{p}{2} \sum_{d \in D_1}d + \frac{1}{2} \sum_{d \in A_2} + \frac{1}{2} \sum_{D_1 \setminus A_1}d = \frac{p}{2} \sum_{d \in D_2}d + \frac{1}{2}\sum_{d \in A_1}d + \frac{1}{2} \sum_{d \in D_2 \setminus A_2}d.$$
    By multiplying this by $2$, we obtain the two-layered partition $\set{(pD_1)\cup A_2 \cup (D_1-A_1), (pD_2)\cup A_1 \cup (D_2-A_2)}$ for $np$, so $np$ is a two-layered number.
    \end{proof}

    \begin{proposition}
    \label{orderdivisors}
    Let the positive divisors of $n$ excluding $1$ be written in increasing order as follows: $ a_1 < a_2 < \dots <a_k =n.$ If $a_{i+1} < 2a_i$ for all $1 \leq i < k$ and $\sigma(n)$ is odd, then $n$ is two-layered.
    \end{proposition}
    
    \begin{proof}
    Let $b_i = a_i$ or $−a_i$ for each $i$. I will explain how to chose the sign of $b_i$ precisely. Then I
    show that $\sum_{i=1}^{k}b_k=0$. Hence, it will imply that $\sigma(n)-1$ can be partitioned into two equal-summed
    subsets.
    
    Let $b_k = a_k = n$ and let $b_{k−1} =−a_{k−1}$. Note that $0 < b_k + b_{k−1} < a_{k−1}$ since $a_k < 2a_{k−1}$. Since the
    current sum $b_k + b_{k−1}$ is positive, we assign the negative sign to $b_{k−2}$. Then $b_{k−2} < b_k +b_{k−1} + b_{k−2} <
    a_{k−1} − a_{k−2} < a_{k−2}$ since $a_{k−1} < 2 a_{k−2}$. If $b_k + b_{k−1} + b_{k−2} \geq 0$, we assign the negative sign to $b_{k−3}$; Otherwise, we assign the positive sign to $b_{k−3}$. Let $s_i$ be $\sum_{j=1}^{k}b_j$. In general, the sign assigned to $b_{i−1}$ is the opposite of the sign of $s_i$ . Let us show inductively that $|s_i | <a_i$ for $1 \leq i \leq k$. It is true for $i =k$. Assume that $|s_{i+1}| < a_{i+1}$. Since the sign of $b_i$ is opposite of the sign of $s_{i+1}$, $|s_i| = ||s_{i+1}| −a_i |$. Note
    that $−a_i < |s_{i+1}| −a_i < a_{i+1} − a_i <a_i$ since $a_{i+1} < 2a_i$ . Therefore $|s_i| < a_i$. So $|s_1| <a_1 = 1$. Since $\sigma(n)-1$
    is even, $s_1$, which is obtained by assigning a positive or negative sign to each of the terms in $\sigma(n)-1$
    is even as well. So $s_1 = 0$. This implies that $\sigma(n)-1$ can be partitioned into two equal-summed subsets. Hence it is two-layered.
    \end{proof}
    
    \begin{proposition} [Proposition 1 in \cite{rao}]
    \label{factorizationsigma}
    Let the prime factorization of $n$ be $\prod_{i=1}^{m}p_i^{k_i}$. Then 
    $$\sigma(n)=\prod_{i=1}^{m}\frac{p_i^{k_i+1}-1}{p_i-1}$$
    and
    $$\frac{\sigma(n)}{n}= \prod_{i=1}^{m} \frac{p_i^{k_i+1}-1}{p_i^{k_i}(p_i-1)} < \prod_{i=1}^{m} \frac{p_i}{p_i-1}$$
    \end{proposition}
    
    \begin{proposition}
    Let the prime factorization of an odd number $n$ be $p_1^k p_2^k \dots p_m^{k_m}$, where $3 \leq p_1 < p_2 < \dots < p_m$. If $n$ is two-layered, then
    $$ \prod_{i=1}^{m} \frac{p_i}{p_i - 1} > 2 ,$$
    and $m$ is at least $3$. In particular:
    \begin{enumerate}
    \item If $m \leq 6$, then $p_1=3$, $p_2=5$, $7$ or $11.$
    \item If $m \leq 4$, then $p_1=3$, $p_2=5$ or 7.
    \item If $m=3$, then $p_1=3$, $p_2=5$, and $p_3=7$ or $11$ or $13$.
    \end{enumerate}
    \end{proposition}
    
    \begin{proof}
    If $n$ is two-layered, then by Propositions \ref{sigmaodd} and \ref{factorizationsigma},
    $$ 2 p_1^{k_1}p_2^{k_2} \dots p_m^{k_m} = 2n < \sigma(n) = \prod_{i=1}^{m}(\sum_{j=0}^{k_i}p_i^j).$$
    Dividing both sides by $p_1^{k_1}p_2^{k_2} \dots p_m^{k_m}$, we get
    $$ 2 < \prod_{i=1}^{m}(\sum_{j=0}^{k_i}p_i^{j-k_i}) < \prod_{i=1}^{m} \frac{p_i}{p_i-1}. $$
    If $m \leq 2$, then
    $$ \prod_{i=1}^{m} \frac{p_i}{p_i-1} \leq \frac{3}{2} \times \frac{5}{4} < 2 $$
    Therefore $m \geq 3$. The parts of $1 - 3$ follows by verifying the condition $\prod_{i=1}^{m} \frac{p_i}{p_i-1} > 2$ directly as given below.
    
    1. Let $m \leq 6$. If $p_1 \neq 3$, then $p_1 \geq 5$ and 
    $$ \prod_{i=1}^{m} \frac{p_i}{p_i-1} \leq \frac{5}{4} \times \frac{7}{6} \times \frac{11}{10} \times \frac{13}{12} \times \frac{17}{16} \times \frac{19}{18} < 2.$$
    Therefore, $p_1 = 3$. If $p_2 > 11$, then $p2 \geq 13$ and
    $$ \prod_{i=1}^{m} \frac{p_i}{p_i-1} \leq \frac{3}{2} \times \frac{13}{12} \times \frac{17}{16} \times \frac{19}{18} \times \frac{23}{22} \times \frac{29}{28} < 2.$$
    Hence, $p_2 \leq 11$. This implies that $p_2 = 5$, $7$ or $11$.
    
    2. Let $m \leq 4$. By $1$, $p_1 = 3$. If $p_2 > 7$, then $p_2 \geq 11$, so
    $$ \prod_{i=1}^{m} \frac{p_i}{p_i-1} \leq \frac{3}{2} \times \frac{11}{10} \times \frac{13}{12} \times \frac{17}{16} < 2.$$
    Therefore, $p_2 \leq 7$. This implies that $p_2 = 5$ or $7$.
    
    3. Let $m = 3$. By $1$, $p_1 = 3$. If $p_2 \neq 5$, then $p_2 \geq 7$ and $p3 \geq 11$. So
    $$ \prod_{i=1}^{3} \frac{p_i}{p_i-1} \leq \frac{3}{2} \times \frac{7}{6} \times \frac{11}{10} < 2.$$
    Hence $p_2 = 5.$
    
    If $p_3 \geq 17$, then
    $$ \prod_{i=1}^{3} \frac{p_i}{p_i-1} \leq \frac{3}{2} \times \frac{5}{4} \times \frac{17}{16} < 2.$$
    Hence, $p_3 < 17$ and consequently $p_3 = 7$, $11$ or $13$.
    \end{proof}

\section{half-layered numbers}
    
    \begin{definition}
    \label{defhalflayered}
    A positive integer $n$ is said to be a half-layered number if the proper positive divisors
    of $n$ excluding $1$ can be partitioned into two disjoint non-empty subsets of equal sum.
    \end{definition}
    
    \begin{definition}
    \label{defhalflayeredpartition}
    A half-layered partition
    for a half-layered number $n$ is a partition $\set {A, B}$ of the set of proper positive divisors of $n$ excluding $1$ so that
    each of $A$ and $B$ sums to the same value.
    \end{definition}
    
    \begin{proposition}
    \label{sigma2n}
    A positive integer $n$ is half-layered if and only if $\frac{\sigma(n)-n-1}{2}$ is the sum of some distinct positive
    proper positive divisors of $n$.
    \end{proposition}
    
    \begin{example}
    In Example \ref{example1}, we saw that $36$ was a two-layered number. It is also a half-layered number and its half-layered partition is $\set{A,B}$, where $A = \set{2, 3, 4, 18}$ and $B=\set {6,9,12}$. You can check that each of $A$ and $B$ has the sum of $27$. The numbers $72, 105,$ and $144$ are also half-layered. You can find the sequence of half-layered numbers in \cite{halfsequenc}.
    
    \end{example}

    \begin{theorem}
    \label{halflay}
    A positive even integer $n$ is half-layered if and only if $\frac{\sigma(n)-2n-1}{2}$ is the sum (possibly empty
    sum) of some distinct positive divisors of $n$ excluding $n$, $\frac{n}{2}$, and $1$.
    \end{theorem}
    
    \begin{proof}
    An even number $n$ is half-layered if and only if there exists a which is the sum (possibly
    empty sum) of some positive divisors of $n$ excluding $n$, $\frac{n}{2}$, and $1$ such that
    $$\frac{n}{2}+a = \frac{\sigma(n)-n-1}{2}.$$
    Therefore, $a = \frac{\sigma(n)-2n-1}{2}.$
    \end{proof}
    
    \begin{example}
    The number $3^4 \times 2^4$ is a half-layered number, since 
    $$\frac{\sigma(3^4 \times 2^4)- 2 ( 3^4 \times 2^4)-1}{2}=579=432+108+36+3$$
    is a sum of positive divisors of $3^4 \times 2^4$ excluding $3^4 \times 2^4$ , $3^4 \times 2^3$, and $1$. Hence, by Theorem \ref{halflay}, it is a half-layered number.
    \end{example}
    
    \begin{proposition}
    \label{oddhalflayered}
    If $n$ is an odd half-layered number, then at least one of the powers of prime factors of $n$ should be even. 
    \end{proposition}
    
    \begin{proof}
    If n is odd and half-layered, then $\sigma(n) − n - 1$ must be even and $\sigma(n)$ must be even. Let the prime factorization of $n$ be $\prod_{i=1}^{m}p_i^{k_i}$. Then $\sigma(n) = \prod_{i=1}^{m}(\sum_{j=0}^{k_i}p_i^j)$. If $\sigma(n)$ is odd, then there exists one $k-i$ which is odd.
    \end{proof}

    \begin{definition} [Definition 3 in \cite{rao}]
    A positive integer $n$ is said to be a half-Zumkeller number if the proper positive divisors of $n$ can be partitioned into two disjoint non-empty subsets of an equal sum. A half-Zumkeller partition for a half-Zumkeller number n is a partition $\set{A, B}$ of the set of proper positive divisors of $n$ so that each of $A$ and $B$ sums to the same value.
    \end{definition}
    
    \begin{proposition}
    \label{mnhalflayered}
    If $m$ and $n$ are half-layered numbers with $(m, n)=1$, then $mn$ is half-layered.
    \end{proposition}
    
    \begin{proof}
    Let $M$ be the set of proper positive divisors of $m$ and let $\set{M_1,M_2}$ be a half-Zumkeller partition
    for $m$. Let $N$be the set of proper positive divisors of $n$ and let $\set{N1, N2}$ be a half-Zumkeller partition
    for $n$. Since $(m,n) = 1$, then the set of proper positive divisors of $mn$ is $(MN) \cup (nM) \cup (mN)$. Observe
    that $\set{(M_1N\setminus \set{1}) \cup (mN_1) \cup (nM_1), (M_2N \setminus \set{1}) \cup (mN_2) \cup (nM_2)}$ is a half-layered partition for $mn$. Therefore $mn$ is
    half-layered.
    \end{proof}
    
    \begin{proposition}
    \label{halfzumnn2}
    Let $n$ be even. Then $n$ is half-layered if and only if $n$ admits a two-layered partition
    such that $n$ and $\frac{n}{2}$ are in distinct subsets. Therefore, if $n$ is an even half-layered number then $n$ is two-layered.
    \end{proposition}
    
    \begin{proof}
    Let $n$ be even. Let $D$ be the set of all positive divisors of $n$ excluding $1$. The number $n$ is half-layered if
    and only if there exists $A \subset D \setminus \set{n, \frac{n}{2}}$ such that
    $$ \frac{n}{2} + \sum_{a \in A}a = \sum_{b \in D, b \not\in \set{n, \frac{n}{2}} \cup A}b.$$
    That is,
    $$ n + \sum_{a \in A}a = \frac{n}{2} + \sum_{b \in D, b \not\in \set{n, \frac{n}{2}} \cup A}b.$$ 
    
    This is equivalent to saying that $n$ admits a two-layered partition such that $n$ and $\frac{n}{2}$ are in distinct
    subsets.
    
    \end{proof}
    
    \begin{theorem}
    \label{itisalsoproved}
    Let $n$ be an even two-layered number. If $\sigma(n) < 3n$, then $n$ is half-layered.
    \end{theorem}
    
    \begin{proof}
    Since $n$ and $\frac{n}{2}$ together sum to more than $\frac{\sigma(n)}{2}$ , they must be in different subsets in any
    two-layered partition for $n$. Therefore, by Proposition \ref{halfzumnn2}, $n$ is half-layered.
    \end{proof}
    
    \begin{proposition}
    \label{sigma3n}
    Let $n$ be even. Then, $n$ is two-layered if and only if either $n$ is half-layered or $\frac{\sigma(n)-3n-1}{2}$ is a sum
    (possibly an empty sum) of some positive divisors of $n$ excluding $n$, $\frac{n}{2}$, and $1$.
    \end{proposition}
    
    \begin{proof}
    Let $n$ be even. If $n$ is two-layered but not half-layered, then by Proposition \ref{halfzumnn2}, any two-layered
    partition of the positive divisors of $n$ must have $n$ and $\frac{n}{2}$ in the same subsets. In other words, there
    exists a which is a sum (possibly an empty sum) of some positive divisors of $n$ excluding $n$, $\frac{n}{2}$, and $1$ such that
    $$ 2(n + \frac{n}{2}+a) = \sigma(n)-1$$
    So, $a= \frac{\sigma(n)-3n-1}{2}$. Therefore, the number $ \frac{\sigma(n)-3n-1}{2}$ is a sum (possibly an empty sum) of some positive
    divisors of $n$ excluding $n$, $\frac{n}{2}$, and $1$.
    
    If $n$ is half-layered, then $n$ is two-layered by Proposition \ref{halfzumnn2}. If $ \frac{\sigma(n)-3n-1}{2}$ is a sum (possibly an empty
    sum) of some positive divisors of $n$ excluding $n$, $\frac{n}{2}$, and $1$, then
    $$ \frac{\sigma(n)-2n-1}{2} = \frac{\sigma(n)- 3n -1}{2} + \frac{n}{2}$$
    is a sum of some positive divisors of $n$ excluding $n$, and $1$. By Theorem \ref{twolayprop}, the number $n$ is two-layered.
    \end{proof}

    \begin{proposition}
    \label{6dividesn}
    If $6$ divides $n$, $n$ is two-layered, and $\sigma(n) < \frac{10n}{3}$ , then $n$ is half-layered.
    \end{proposition}
    
    \begin{proof}
    If $n$ is not half-layered, by Proposition \ref{sigma3n}, $\frac{\sigma(n)-3n-1}{2}$ is a sum (might be an empty sum)
    of some positive divisors of $n$ excluding $n$, $\frac{n}{2}$, and $1$. Then,
    $$ \frac{\sigma(n)-2n-1}{2} = \frac{\sigma(n)- 3n -1}{2} + \frac{n}{3} + \frac{n}{6}.$$
    Since $\sigma(n)/n < \frac{10}{3}$ we have that $\frac{\sigma(n)-3n-1}{2} < \frac{n}{6}$. Hence $\frac{\sigma(n)-2n-1}{2}$ is a sum of some positive
    divisors of $n$ excluding $n$, $\frac{n}{2}$, and $1$. By Proposition \ref{sigma2n}, $n$ is half layered. This is a contradiction.
    \end{proof}
    
    \begin{proposition}
    \label{n2ntwolayered}
    If $n$ is two-layered, then $2n$ is half-layered.
    \end{proposition}
    
    \begin{proof}
    Let $n = 2^k L$ with $k$ a nonnegative integer and $L$ an odd number, be a two-layered number. Then
    all positive divisors of $n$ excluding $1$ can be partitioned into two disjoint equal-summed subsets $D_1$ and $D_2$.
    Observe that every positive divisor of $2n$ which is not a positive divisor of $n$ can be written as $2^{k+1} \ell$
    where $\ell$ is a positive divisor of $L$. Observe that $2^k \ell$ is either in $D_1$ or $D_2$. Without loss of generality,
    assume that $2^k \ell$ is in $D_1$. In this case, we move $2^k \ell$ to $D_2$ and add $2^{k+1} \ell$ to $D_1$. Perform this procedure
    to all positive divisors of $2n$ which are not positive divisors of $n$ except $2n$ itself. This procedure will
    yield an equal-summed partition of all positive divisors of $2n$ except $2n$ itself. This shows that $2n$ is
    half-Zumkeller.
    \end{proof}
    
    \begin{corollary}
    Let $n$ be even and the prime factorization of $n$ be $2^k p_1^{k_1} \dots p_m^{k_m}$. If $n$ is two-layered but not half-
    layered, then $2^i p_1^{k_1} \dots p_m^{k_m}$ is not two-layered for any $i \leq k-1$, and $2^i p_1^{k_1} \dots p_m^{k_m}$ is half-layered for any
    $i \geq k + 1$. 
    \end{corollary}
    
    \begin{proposition}
    \label{andprimefactorization}
    Let $n$ be an even half-layered number and $p$ be a prime with (n, p) = 1. Then $np^{\ell}$ is half-
    layered for any positive integer $\ell$.
    \end{proposition}
    
    \begin{proof}
    Since $n$ is an even half-layered number, the set of all positive divisors of $n$, excluding $1$, denoted by $D_0$
    can be partitioned into two disjoint subsets $A_0$ and $B_0$ so that the sums of the two subsets are equal
    and $n$ and $\frac{n}{2}$ are in distinct subsets (by Proposition \ref{halfzumnn2}).
    
    Group the positive divisors of $np^{\ell}$ except $1$ into $\ell +1$ groups $D_0, D_1,\dots D_{\ell}$ according to how many positive
    divisors of $p$ they admit, i.e., $D_i$ consists of all positive divisors of $np^{\ell}$ admitting $i$ positive divisors
    of $p$. Then each $D_i$ can be partitioned into two disjoint subsets so that the sums of the two subsets
    are equal and $np^i$ and $\frac{np^i}{2}$ are in distinct subsets according to the two-layered partition of the set $D_0$. Therefore all positive divisors of $np^{\ell}$ excluding $1$ can be partitioned into two disjoint subsets so that the sum
    of these two subsets equal and $np^{\ell}$ and $\frac{np^{\ell}}{2}$ are in distinct subsets. By Proposition \ref{halfzumnn2}, $np^{\ell}$ is half-
    layered.
    \end{proof}
    
    \begin{corollary}
    If $n$ is an even half-layered number and $m$ is a positive integer with $(n,m) = 1$, then $nm$ is
    half-layered.
    \end{corollary}
    
    \begin{theorem}
    \label{numberandtheprime}
    Let $n$ be an even half-layered number and the prime factorization of $n$ be $ p_1^{k_1} p_2^{k_2} /dots p_m^{k_m} $ Then for nonnegative integers $\ell_1, \dots , \ell_m$, the integer
    $$ p_1^{k_1+\ell_1(k_1+1)} p_2^{k_2+\ell_2(k_2+1)} \dots p_m^{k_m+\ell_m(k_m+1)}    $$
    is half-layered.
    \end{theorem}

    \begin{proof}
    It is sufficient to show that $ p_1^{k_1+\ell_1(k_1+1)} p_2^{k_2} \dots p_m^{k_m} $ is half-layered if $ p_1^{k_1} p_2^{k_2} \dots p_m^{k_m} $ is an
    even half-layered number. Assume that $ p_1^{k_1} p_2^{k_2} \dots p_m^{k_m} $ is even and half-layered, then the set
    of all positive divisors of $n$ excluding $1$, denoted by $D_0$ can be partitioned into two disjoint subsets $A_0$ and $B_0$ so
    that the sums of the two subsets are equal and $n$ and $\frac{n}{2}$ are in distinct subsets (by Proposition \ref{halfzumnn2}).
    Note that the positive divisors of $ p_1^{k_1+\ell_1(k_1+1)} p_2^{k_2} \dots p_m^{k_m} $ excluding $1$ can be partitioned into $\ell_1 + 1$ disjoint groups
    $D_i, 0 \leq i \leq \ell_1$, where elements in $D_i$ are obtained by multiplying $p_1^{i(k_1+1)}$ with elements in $D_0$. Using
    the partition $A_0, B_0$ of $D0$ we can partition every $D_i$ into two disjoint subsets $A_i$ and $B_i$ so that
    the sums of the corresponding subsets are equal and $ np_1^{i(k_1)+1}$ and $\frac{np_1^{i(k_1)+1}}{2}$ are in distinct subsets.
    Therefore, the set of all positive divisors of $ p_1^{k_1+\ell_1(k_1+1)} p_2^{k_2} \dots p_m^{k_m} $ excluding $1$ can be partitioned into two disjoint equal-summed subsets and $ p_1^{k_1+\ell_1(k_1+1)} p_2^{k_2} \dots p_m^{k_m} $ and $\frac{p_1^{k_1+\ell_1(k_1+1)} p_2^{k_2} \dots p_m^{k_m}}{2} $ are in distinct subsets. By
    Proposition \ref{halfzumnn2}, $ p_1^{k_1+\ell_1(k_1+1)} p_2^{k_2} \dots p_m^{k_m} $ is half-layered.
    \end{proof}
    
    \begin{theorem}
    Let $n$ be an even integer and $p$ be a prime with $(n, p) = 1$. Let $D $ be the set of all positive
    divisors of $n$ excluding $1$. Then the following conditions are equivalent:
    \begin{enumerate}
    \item $np$ is half-layered.
    \item There exist two partitions $\set{D_1, D_2}$ and $\set{D_3, D_4}$ of $D$ such that $n$ is in $D_1$, $\frac{n}{2}$ is in $D_2$ and
    $$ p (\sum_{d \in D_1}d - \sum_{d \in D_2} d) =  \sum_{d \in D_3}d - \sum_{d \in D_4}d. $$
    \item There exists a partition $\set{D_1, D_2}$ of $D$ and subsets $A_1 \subseteq D_1$and $A_2 \subseteq D_2$ such that $n$ is in $D_1$, $\frac{n}{2}$ is in $D_2$
    and
    $$ \frac{p+1}{2} (\sum_{d \in D_1}d - \sum_{d \in D_2} d) =  \sum_{d \in A_1}d - \sum_{d \in A_2}d.  $$
    \end{enumerate}
    \end{theorem}
    
    \begin{proof}
    By Proposition \ref{halfzumnn2}, $np$ is half-layered if and only if there is a two-layered partition $\set{A, B}$ of
    $(pD) \cup D $such that $n \in A$ and  $\frac{n}{2} \in B$. The rest of the proof follows along the lines of the proof of
    Theorem \ref{somesum}.
    \end{proof}
    
    \begin{proposition}
    If $a_1 < a_2 < \dots < a_k = n$ are all positive divisors of an even number $n$ excluding $1$ with $a_{i+1} < 2a_i$
    for all $i$ and $\sigma(n)$ is odd, then $n$ is half-layered.
    \end{proposition}

    \begin{proof}
    Note that in the proof of Proposition \ref{orderdivisors}, $b_k = n$ and $b_{k−1} = - \frac{n}{2}$ have different signs. So we
    get a two-layered partition of $n$ such that $n$ and $\frac{n}{2}$ are in distinct subsets. By Proposition \ref{halfzumnn2}, $n$ is
    half-layered.
    \end{proof}


\begin{thebibliography}{9}
    
    \bibitem{zumkeller}
    The online Encyclopedia of Integer Sequences, https://oeis.org/A083207/.
    
    
    \bibitem{clark}
    S. Clark, J. Dalzell, J. Holliday, D. Leach, M. Liatti, M. Walsh, Zumkeller numbers, presented in the Mathematical Abundance
    Conference at Illinois State University on April 18th, 2008.
    
    \bibitem{rao}
    K. P. S. Bhaskara Rao and Yuejian Peng, On Zumkeller numbers, Journal of Number Theory 133, No. 4
    (2013) 1135-1155
    
    \bibitem{twosequenc}
    The online Encyclopedia of Integer Sequences, https://oeis.org/A322657/.
     
    \bibitem{practical}
    A.K. Srinivasan, Practical numbers, Current Sci. 17 (1948) 179–180, MR0027799.
    
    \bibitem{halfsequenc}
    The online Encyclopedia of Integer Sequences, https://oeis.org/A322658/.
    
        
    \end{thebibliography}
\end{document}